\documentclass[11pt]{amsart}

\usepackage[margin=1.25in]{geometry}

\usepackage{graphicx}

\usepackage{amsmath, amssymb, amsthm, graphicx, enumerate, tikz, float}
\usepackage{mathtools, comment}
\usepackage{wasysym}
\usepackage[misc]{ifsym}
\usepackage[colorlinks]{hyperref}
\hypersetup{citecolor=blue}
\usetikzlibrary{matrix,arrows,decorations.pathmorphing}

\newtheorem{theorem}{Theorem}[section]
\newtheorem{lemma}[theorem]{Lemma}
\newtheorem{corollary}[theorem]{Corollary}
\newtheorem{proposition}[theorem]{Proposition}
\newtheorem*{thmmain}{Theorem 1.1}

\theoremstyle{definition}
\newtheorem{definition}[theorem]{Definition}

\newtheorem{hyp}[theorem]{Hypothesis}

\theoremstyle{remark}
\newtheorem{remark}[theorem]{Remark}

\DeclareMathOperator{\cd}{cd}
\DeclareMathOperator{\Irr}{Irr}

\numberwithin{equation}{section}

\newcommand{\SL}[2]{\Sigma_{#1,#2}^L} 
\newcommand{\SR}[2]{\Sigma_{#1,#2}^R} 
\newcommand{\SLR}[2]{\Sigma_{#1,#2}^*} 

\begin{document}

\allowdisplaybreaks

\title[prime character degree graphs within a family]{On prime character degree graphs occurring within a family of graphs}

\author{Jacob Laubacher}
\address{Department of Mathematics, St. Norbert College, De Pere, WI 54115}
\email{jacob.laubacher@snc.edu}

\author{Mark Medwid}
\address{Department of Mathematical Sciences, Rhode Island College, Providence, RI 02908}
\email{mmedwid@ric.edu}


\date{\today}

\keywords{character degree graphs, solvable groups, families of graphs\\\indent\emph{Corresponding author.} Jacob Laubacher \Letter~\href{mailto:jacob.laubacher@snc.edu}{jacob.laubacher@snc.edu} \phone~920-403-2961.}

\begin{abstract}
In this paper we investigate families of connected graphs which do not contain an odd cycle in their complement. 
Specifically, we consider graphs formed by two complete graphs connected in a particular way.
We determine which of these graphs can or cannot occur as the prime character degree graph of a solvable group. An obvious expansion and generalization can also be considered, of which we make mention.
\end{abstract}

\maketitle

\section{Introduction}

Let $G$ be a finite solvable group. As is conventional, we set $\Irr(G)$ as the irreducible characters of $G$, and $\cd(G)=\{\chi(1)~:~\chi\in\Irr(G)\}$. The prime character degree graph, denoted $\Delta(G)$, has the vertex set $\rho(G)$, which is the set of all prime divisors of $\cd(G)$. There is an edge between two distinct vertices $p$ and $q$ if there exists some $a\in\cd(G)$ such that $pq$ divides $a$. When discussing $\Delta(G)$, often times we will interchange the words prime and vertex, as their meaning is synonymous in this context.

Prime character degree graphs have been studied for some time (see \cite{H}, \cite{I}, \cite{L3}, \cite{MW}, or \cite{Z} for a broad overview). In this setting, given a simple undirected graph $\Gamma$, one of the main questions that one can ask when investigating prime character degree graphs is whether or not there exists some solvable group $G$ such that $\Gamma=\Delta(G)$. P\'alfy's condition, most notably, helps to begin to answer this question by declaring that for any three vertices in $\Delta(G)$ there must be at least one edge incident between two of them. Another way to view this landmark result from \cite{P} is that there is no triangle (or 3-cycle) in the complement graph $\overline{\Delta}(G)$. P\'alfy's condition was generalized in 2018 in \cite{etal} to say that $\overline{\Delta}(G)$ must contain no odd cycle. These results about the complement are extremely useful when trying to eliminate graphs, as seen in the classifications for graphs with five and six vertices in \cite{Lewis} and \cite{BL2}, respectively.


Some work has been done with regards to families of graphs. In \cite{BL}, Bissler and Lewis constructed a family of graphs that do not occur as the prime character degree graph of any solvable group. The authors make use of P\'alfy's condition where necessary; they also introduce an additional notion called \emph{admissibility} of a vertex. Informally, admissible vertices of a graph are those that are not solely responsible for preventing the graph from manifesting as the prime character degree graph of a solvable group. Bissler and Lewis showed that a graph consisting of only admissible vertices cannot appear as the prime character degree graph of a solvable group. Some particular families of graphs are also considered and classified in \cite{BissLaub}; one goal of this paper is to expand the construction and results of \cite{BissLaub} in a natural way. 


Another motivation for this paper is to provide more examples of graphs satisfying the technical hypotheses found in \cite{BLL}. To that end, in Section \ref{gfamilies}, we introduce the two graph families $\{\SL{k}{n}\}$ and $\{\SR{k}{n}\}$. Both families are constructed with a newly modified version of the construction from \cite{BissLaub}. Save for one alteration, the two families are constructed identically. Graphs in both families consist of two distinct, non-overlapping subgraphs: a ``left'' subgraph and ``right'' subgraph. The left subgraph is a complete graph on $k$ vertices, while the right subgraph is a complete graph on $k+n$ vertices. The choice of $n$ also dictates what the edges connecting the left and right subgraphs look like: $n$ vertices in the left subgraph are connected in a one-to-two fashion to particular vertices in the right subgraph, while the remainder of vertices in the left subgraph are connected to the right subgraph in a one-to-one fashion according to their indices. 

The distinction of the two families stems from the addition of a special vertex to either the left or right subgraph (hence, an $L$ or $R$ superscript). This special vertex is adjacent to every vertex in its chosen subgraph, but is \emph{not} adjacent to any vertices in the other subgraph. As it turns out, this alteration affects the resulting graph's structure profoundly enough that both families need to be treated separately. 

Section \ref{first} is dedicated to studying the family $\{\SL{k}{n}\}$. For all integers $1\leq n\leq k$ we completely classify which graphs occur or do not occur as $\Delta(G)$ for some finite solvable group $G$. We do this through induction on $n$. Most of the details provided thoroughly go through the base case when $n=1$, while the inductive step is a natural extension from there. One can sum up the section, along with our main result for this paper, as follows: 

\begin{theorem}\label{firstfam}
The graph $\SL{k}{n}$ occurs as the prime character degree graph of a solvable group only when $(k,n)=(1,1)$ (see Figure \ref{fig11L}). 
\end{theorem}

In Section \ref{last} we discuss further generalizations that could be considered. We address looking at one-to-$m$ edge mappings, for any natural number $m$, as well as increasing the number of special vertices on the left or right. Finally, for the sake of length and argument, we leave the family $\{\SR{k}{n}\}$ for a follow-up paper.

\section{Preliminaries}

In this section, we present a brief survey of various results and technical hypotheses that are useful in the proof of Theorem \ref{firstfam}. Many of the results used in this paper are classical, such as P\'alfy's condition from \cite{P}. We also depend upon more contemporary results, such as those found in \cite{etal}, \cite{BLL}, \cite{BL}, and \cite{Sass}.

We shall first touch upon work that has been done with regards to classifying families of graphs vis-\'a-vis their appearance as the prime character degree graphs of solvable groups. We then look at three related tools that can be used in such classifications: results on disconnected graphs, the absence of normal nonabelian Sylow subgroups, and character degree graphs of diameter three. 

\subsection{Families of graphs}

Families of graphs have been shown not to occur as the prime character degree graph of solvable groups for some time. Most notably, P\'alfy's condition below gives us a good starting point.

\begin{lemma}[P\'alfy's condition from \cite{P}]\label{PC}
Let $G$ be a solvable group and let $\pi$ be a set of primes contained in $\Delta(G)$. If $|\pi|=3$, then there exists an irreducible character of $G$ with degree divisible by at least two primes from $\pi$. (In other words, any three vertices of the prime character degree graph of a solvable group span at least one edge.)
\end{lemma}

More recently in \cite{etal}, a generalization of P\'alfy's condition was brought to light.

\begin{theorem}\label{et}\emph{(\cite{etal})}
Let G be a finite solvable group. Then the complement graph $\overline{\Delta}(G)$ does not contain any cycle of odd length.
\end{theorem}

Next, we recall the construction of the family of graphs from \cite{BissLaub}. Take complete graphs $A$ and $B$, where $A$ has $k$ vertices, $a_1,a_2,\ldots,a_k$, and $B$ has $t$ vertices, $b_1,b_2,\ldots,b_t$. Without loss of generality, say that $k\geq t$. Let $\rho(\Gamma_{k,t})=\rho(A)\cup\rho(B)$. Since $\rho(A)$ and $\rho(B)$ can be chosen as disjoint sets, one gets $|\rho(\Gamma_{k,t})|=k+t$. There is an edge between vertices $p$ and $q$ in $\Gamma_{k,t}$ if any of the following are satisfied:
\begin{enumerate}[(i)]
    \item $p,q\in\rho(A)$,
    \item $p,q\in\rho(B)$, or
    \item $p=a_i\in\rho(A)$ and $q=b_i\in\rho(B)$ for some $1\leq i\leq t$.
\end{enumerate}

\begin{theorem}\label{KT}\emph{(\cite{BissLaub})}
The graph $\Gamma_{k,t}$ occurs as the prime character degree graph of a solvable group precisely when $t=1$ or $k=t=2$. Otherwise $\Gamma_{k,t}$ does not occur as the prime character degree graph of any solvable group.
\end{theorem}

\begin{corollary}\label{subKT}\emph{(\cite{BissLaub})}
Let $k\geq t\geq2$. No connected proper subgraph of $\Gamma_{k,t}$ with the same vertex set is the prime character degree graph of any solvable group.
\end{corollary}

Finally, there are tools from \cite{BL} that help us determine if a graph (that satisfies P\'alfy's condition) does in fact not occur as $\Delta(G)$. This result is below:

\begin{lemma}\label{lemma3}\emph{(\cite{BL})}
Let $\Gamma$ be a graph satisfying P\'alfy's condition with $n\geq5$ vertices. Also, assume there exist distinct vertices $a$ and $b$ of $\Gamma$ such that $a$ is adjacent to an admissible vertex $c$, $b$ is not adjacent to $c$, and $a$ is not adjacent to an admissible vertex $d$.

Let $G$ be a solvable group and suppose for all proper normal subgroups $N$ of $G$ we have that $\Delta(N)$ and $\Delta(G/N)$ are proper subgraphs of $\Gamma$. Let $F$ be the Fitting subgroup of $G$ and suppose that $F$ in minimal normal in $G$. Then $\Gamma$ is not the prime character degree graph of any solvable group.
\end{lemma}

\subsection{Disconnected graphs}

It is easy to verify that if a disconnected graph does occur as the prime character degree graph of a solvable group, then there must be exactly two components, and each component must be complete. This is due to P\'alfy's condition. Moreover, in \cite{L2}, Lewis completely classifies the disconnected graphs with two connected components. One of the most powerful tools, however, is another result by P\'alfy:

\begin{theorem}[P\'alfy's inequality from \cite{P2}]
Let $G$ be a solvable group and $\Delta(G)$ its prime character degree graph. Suppose that $\Delta(G)$ is disconnected with two components having size $a$ and $b$, where $a\leq b$. Then $b\geq2^a-1$.
\end{theorem}

Moreover, we will occasionally use the following result by Lewis whenever a disconnected subgraph arises.

\begin{theorem}[Theorem 5.5 from \cite{L2}]
Let $G$ be a solvable group and suppose that $\Delta(G)$ has two connected components. Then there is precisely one prime $p$ so that the Sylow $p$-subgroup of the Fitting subgroup of $G$ is not central in $G$.
\end{theorem}

\subsection{Normal nonabelian Sylow subgroups}

One of the most useful tools in \cite{BL} was to classify vertices in a graph as admissible. This had far-reaching consequences, many of which will not be discussed in this paper.

\begin{definition}(\cite{BL})
A vertex $p$ of a graph $\Gamma$ is \textbf{admissible} if:
\begin{enumerate}[(i)]
    \item the subgraph of $\Gamma$  obtained by removing $p$ and all edges incident to $p$ does not occur as the prime character degree graph of any solvable group, and
    \item none of the subgraphs of $\Gamma$ obtained by removing one or more of the edges incident to $p$ occur as the prime character degree graph of any solvable group.
\end{enumerate}
\end{definition}

\begin{lemma}\emph{(\cite{BL})}
Let $G$ be a solvable group, and suppose $p$ is an admissible vertex of $\Delta(G)$. For every proper normal subgroup $N$ of $G$, suppose that $\Delta(N)$ is a proper subgraph of $\Delta(G)$. Then $O^p(G)=G$.
\end{lemma}

A refinement was made in \cite{BL} to the definition of admissible.

\begin{definition}(\cite{BL})
A vertex $p$ of a graph $\Gamma$ is \textbf{strongly admissible} if:
\begin{enumerate}[(i)]
    \item $p$ is admissible, and
    \item none of the subgraphs of $\Gamma$ obtained by removing $p$, the edges incident to $p$, and one or more of the edges between two adjacent vertices of $p$ occurs as $\Delta(G)$ for some solvable group $G$.
\end{enumerate}
\end{definition}

\begin{lemma}\label{strong}\emph{(\cite{BL})}
Let $G$ be a solvable group and assume that $p$ is a prime whose vertex is a strongly admissible vertex of $\Delta(G)$. For every proper normal subgroup $N$ of $G$, suppose that $\Delta(G/N)$ is a proper subgraph of $\Delta(G)$. Then a Sylow $p$-subgroup of $G$ is not normal.
\end{lemma}

This is the kind of conclusion that we desire: no normal nonabelian Sylow subgroup. The following lemma gives an alternative (yet distinct) method to reach this result.

\begin{lemma}\label{pi}\emph{(\cite{BL})}
Let $\Gamma$ be a graph satisfying P\'alfy's condition. Let $q$ be a vertex of $\Gamma$, and denote $\pi$ to be the set of vertices of $\Gamma$ adjacent to $q$, and $\rho$ to be the set of vertices of $\Gamma$ not adjacent to $q$. Assume that $\pi$ is the disjoint union of nonempty sets $\pi_1$ and $\pi_2$, and assume that no vertex in $\pi_1$ is adjacent in $\Gamma$ to any vertex in $\pi_2$. Let $v$ be a vertex in $\pi_2$ adjacent to an admissible vertex $s$ in $\rho$. Furthermore, assume there exists another vertex $w$ in $\rho$ that is not adjacent to $v$.

Let $G$ be a solvable group such that $\Delta(G)=\Gamma$, and assume that for every proper normal subgroup $N$ of $G$, $\Delta(N)$ is a proper subgraph of $\Delta(G)$. Then a Sylow $q$-subgroup of $G$ for the prime associated to $q$ is not normal.
\end{lemma}

We now present a third and final approach to determine that there is no normal nonabelian Sylow subgroup for some vertex $p$. This is much more technical, and finding examples of such graphs that satisfy the below hypothesis was the original motivation of this paper.

Following \cite{BLL}, we start by considering a finite solvable group $G$ where $|G|$ is minimal and $\Delta(G)$ is its corresponding prime character degree graph. For an arbitrary vertex $p\in\rho(G)$, let $\pi$ consist of all the vertices adjacent to $p$, and let $\rho$ consist of all the vertices nonadjacent to $p$. Let $\pi^*$ and $\rho^*$ denote arbitrary nonempty subsets of $\pi$ and $\rho$, respectively. Finally, let $\pi^*\cup\rho^*$ be an arbitrary vertex set which induces a complete subgraph in $\Delta(G)$. Let $\beta$ be a subset of $\rho(G)$ that contains $\pi^*\cup\rho^*$ such that $\beta$ also induces a complete subgraph in $\Delta(G)$. Set $\mathcal{B}$ to be the union of all such $\beta$'s satisfying these properties for $\pi^*\cup\rho^*$. Consider $\tau:=\mathcal{B}\setminus(\pi^*\cup\rho^*)$, and denote $\tau^*$ to be a subset of $\tau$. Notice that $\tau$ could be empty, depending on the initial set $\pi^*\cup\rho^*$.

\begin{hyp}(\cite{BLL})\label{newhype}
Concerning $\Delta(G)$, we assume the following:
\begin{enumerate}[(i)]
    \item\label{11} for every vertex in $\rho$, there exists a nonadjacent vertex in $\pi$,
    \item\label{22} for every vertex in $\pi$, there exists a nonadjacent vertex in $\rho$,
    \item\label{33} all the vertices in $\pi$ are admissible. Moreover, no proper connected subgraph with vertex set $\{p\}\cup\pi^*\cup\rho$ occurs as the prime character degree graph of any solvable group,
    \item\label{44} for each vertex set $\pi^*\cup\rho^*$ which induces a complete subgraph in $\Delta(G)$, all the vertices in the corresponding set $\tau$ are admissible. Moreover, no proper connected subgraph with vertex set $\rho(G)\setminus\tau^*$ occurs as the prime character degree graph of any solvable group, and
    \item\label{55} if a disconnected subgraph with vertex set $\rho(G)$ does not occur, then it must specifically violate P\'alfy's inequality from \cite{P2}. Finally, if a disconnected subgraph with vertex set $\rho(G)$ does occur, then the sizes of the connected components must be $n>1$ and $2^n-1$.
\end{enumerate}
\end{hyp}

\begin{theorem}\emph{(\cite{BLL})}\label{newmain}
Assume Hypothesis \ref{newhype}. Then $G$ has no normal nonabelian Sylow $p$-subgroup.
\end{theorem}

Finally we present a useful result should there be a normal Sylow $p$-subgroup. The Lemma below tells us about the vertex set of the resulting graph, whereas the edges that could be there are more delicate.

\begin{lemma}\label{lewisgraph}\emph{(\cite{L})}
Let $G$ be a solvable group and let $p\in\rho(G)$. If $P$ is a normal Sylow $p$-subgroup of $G$, then $\rho(G/P')=\rho(G)\setminus\{p\}$.
\end{lemma}

Concerning Lemma \ref{lewisgraph}, the vertex $p$ is removed, along with all incident edges. Furthermore, edges between vertices that are adjacent to $p$ may also be lost.

\subsection{Graphs with diameter three}

The following was presented in \cite{Sass}. Suppose $\Gamma$ is a graph of diameter three. We can partition $\rho(\Gamma)$ into four nonempty disjoint sets: $\rho(\Gamma)=\rho_1\cup\rho_2\cup\rho_3\cup\rho_4$. One can do this in the following way: find vertices $p$ and $q$ where the distance between them is three. Let $\rho_4$ be the set of all vertices that are distance three from the vertex $p$, which will include the vertex $q$. Let $\rho_3$ be the set of all vertices that are distance two from the vertex $p$. Let $\rho_2$ be the set of all vertices that are adjacent to the vertex $p$ and some vertex in $\rho_3$. Finally, let $\rho_1$ consist of $p$ and all vertices adjacent to $p$ that are not adjacent to anything in $\rho_3$. This notation is not unique, and depends on the vertices $p$ and $q$. Using the above notation, one can always arrange the four disjoint sets to have the following result.

\begin{proposition}\label{sassresult}\emph{(\cite{Sass})}
Let $G$ be a solvable group where $\Delta(G)$ has diameter three. One then has the following:
\begin{enumerate}[(i)]
    \item\label{sass1} $|\rho_3|\geq3$,
    \item\label{sass2} $|\rho_1\cup\rho_2|\leq|\rho_3\cup\rho_4|$,
    \item\label{sass3} if $|\rho_1\cup\rho_2|=n$, then $|\rho_3\cup\rho_4|\geq2^n$, and
    \item\label{sass4} $G$ has a normal Sylow $p$-subgroup for exactly one prime $p\in\rho_3$.
\end{enumerate}
\end{proposition}

\section{Families of Graphs} \label{gfamilies}

In this section we introduce the two families of graphs that we consider throughout this paper: $\{\SL{k}{n}\}$ and $\{\SR{k}{n}\}$.We shall use $\{\SLR{k}{n}\}$ to refer to both families simultaneously. Both $\SL{k}{n}$ and $\SR{k}{n}$ are constructed similarly. Let $k$ and $n$ be integers satisfying $1 \leq n \leq k$. The graph $\SLR{k}{n}$ consists of two distinct subgraphs $A$ and $B$, a fixed vertex $c$, and obeys the following:
\begin{enumerate}[(i)]
    \item $A$ is a complete graph on $k$ vertices $a_1, a_2, \ldots, a_k$,
    \item $B$ is a complete graph on $k+n$ vertices $b_1,b_2, \ldots, b_k, \ldots, b_{k+n}$,
    \item $c \notin \rho(A)$ and $c \notin \rho(B)$,
    \item $\rho(A) \cap \rho(B) = \varnothing$,
    \item there is an edge between $a_i$ and $b_i$ for all $1 \leq i \leq k$,
    \item there is an edge between $a_i$ and $b_{k+i}$ for all $1 \leq i \leq n$,
    \item there is an edge between $c$ and $a_i$ for all $1 \leq i \leq k$ in $\SL{k}{n}$,
    \item there is an edge between $c$ and $b_i$ for all $1 \leq i \leq k+n$ in $\SR{k}{n}$,
    \item there are no edges in the graph $\SLR{k}{n}$ other than the edges described in (v) - (viii).
\end{enumerate}

At a birds-eye view of the graph $\SLR{k}{n}$, we can see two distinct complete subgraphs: $A$ (the ``left" subgraph) and $B$ (the ``right" subgraph). The left subgraph, $A$, always has $k$ vertices; the first $n$ of these vertices are adjacent to precisely two vertices in $B$. So, despite the number of conditions to obey, construction of the graph $\SLR{k}{n}$ is very simple in practice.

First, one chooses a positive integer $k$ and starts with a complete graph on $k$ vertices. The next choice is $n$, which is the number of vertices in $A$ attached in a one-to-two fashion to vertices in $B$. Any remaining vertex of $A$ simply has one edge between it and the vertex of corresponding index in $B$. Keeping in mind that $B$ is a complete graph, it is easy to construct $B$ given $A$ and $n$. The final decision is where the vertex $c$ ends up -- with the left graph $A$ (as in $\SL{k}{n}$) or the right graph $B$ (as in $\SR{k}{n}$). Adding in the necessary edges between $c$ and the vertices of $A$ (respectively, $B$), the subgraph with vertices $\rho(A) \cup \{c\}$ (respectively, $\rho(B) \cup \{c\}$) and all the associated edges yields a complete graph on $k+1$ (respectively, $k+n+1$) vertices. The vertex $c$ is special in the sense that it is ``one-sided." It is only adjacent to vertices in $A$ or $B$, but not both. All other vertices of  $\SLR{k}{n}$  are adjacent to vertices in both $A$ and $B$.

\begin{remark}\label{rmk}
The graphs in $\{\SLR{k}{n}\}$ arise as a natural extension of the graphs $\Gamma_{k,t}$ in \cite{BissLaub}. As in \cite{BissLaub}, the graphs constructed here are connected graphs of diameter two and hence are worthy of study -- this is because graphs of diameter one are required to be complete, and graphs of diameter three have been substantially tamed in \cite{Sass}. The graphs $\SLR{k}{n}$ differ from $\Gamma_{k,t}$ in two important aspects. First, some (or perhaps all) of the edges in the subgraph $A$ are connected to $B$ in a one-to-two fashion, whereas the vertices in the analogous subgraphs of $\Gamma_{k,t}$ are always connected to one another in an injective fashion. Second is the addition of the special vertex $c$ that is incident only to the subgraph $A$ (in $\SL{k}{n}$) or the subgraph $B$ (in $\SR{k}{n}$). Such points arise in $\Gamma_{k,t}$ when $t$ is strictly less than $k$ but are always associated to the subgraph with more vertices. However, in the construction of $\SLR{k}{n}$, we are free to choose from the two subgraphs to which we will associate the vertex $c$. This choice alters the graphs substantially enough that both $\SL{k}{n}$ and $\SR{k}{n}$ need to be considered separately. Finally, we observe that it is not worthwhile for us to investigate this special vertex $c$ existing on the left and right simultaneously, for if that occurred, the graph would have diameter three. 
\end{remark}

One can see Figure \ref{figL} for some examples of graphs in the family $\{\SL{k}{n}\}$ and Figure \ref{figR} for examples of graphs in the family $\{\SR{k}{n}\}$.
\begin{figure}[htb]
    \centering
$
\begin{tikzpicture}[scale=2]
\node (1b) at (.5,1) {$a_1$};
\node (2b) at (.5,0) {$a_2$};
\node (3b) at (1,.5) {$a_3$};
\node (11b) at (2.25,1) {$b_1$};
\node (22b) at (2.25,0) {$b_2$};
\node (33b) at (1.75,.5) {$b_3$};
\node (44b) at (2.75,.75) {$b_4$};
\node (55b) at (2.75,.25) {$b_5$};
\node (66b) at (0,.5) {$c$};
\path[font=\small,>=angle 90]
(1b) edge node [right] {$ $} (2b)
(1b) edge node [right] {$ $} (3b)
(2b) edge node [right] {$ $} (3b)
(66b) edge node [right] {$ $} (1b)
(66b) edge node [right] {$ $} (2b)
(66b) edge node [right] {$ $} (3b)
(11b) edge node [right] {$ $} (22b)
(11b) edge node [right] {$ $} (33b)
(11b) edge node [right] {$ $} (44b)
(11b) edge node [right] {$ $} (55b)
(22b) edge node [right] {$ $} (33b)
(22b) edge node [right] {$ $} (33b)
(22b) edge node [right] {$ $} (44b)
(22b) edge node [right] {$ $} (55b)
(33b) edge node [right] {$ $} (44b)
(33b) edge node [right] {$ $} (55b)
(44b) edge node [right] {$ $} (55b)
(1b) edge node [right] {$ $} (11b)
(2b) edge node [right] {$ $} (22b)
(3b) edge node [right] {$ $} (33b)
(1b) edge node [right] {$ $} (44b)
(2b) edge node [right] {$ $} (55b);
\node (0f) at (3.5,.5) {$c$};
\node (1f) at (4,1) {$a_1$};
\node (2f) at (4,0) {$a_2$};
\node (3f) at (4.5,.5) {$a_3$};
\node (11f) at (5.75,1) {$b_1$};
\node (22f) at (5.75,0) {$b_2$};
\node (33f) at (5.25,.75) {$b_3$};
\node (44f) at (6.25,.75) {$b_4$};
\node (55f) at (6.25,.25) {$b_5$};
\node (66f) at (5.25,.25) {$b_6$};
\path[font=\small,>=angle 90]
(0f) edge node [right] {$ $} (1f)
(0f) edge node [right] {$ $} (2f)
(0f) edge node [right] {$ $} (3f)
(1f) edge node [right] {$ $} (2f)
(1f) edge node [right] {$ $} (3f)
(2f) edge node [right] {$ $} (3f)
(11f) edge node [right] {$ $} (22f)
(11f) edge node [right] {$ $} (33f)
(11f) edge node [right] {$ $} (44f)
(11f) edge node [right] {$ $} (55f)
(11f) edge node [right] {$ $} (66f)
(22f) edge node [right] {$ $} (33f)
(22f) edge node [right] {$ $} (44f)
(22f) edge node [right] {$ $} (55f)
(22f) edge node [right] {$ $} (66f)
(33f) edge node [right] {$ $} (44f)
(33f) edge node [right] {$ $} (55f)
(33f) edge node [right] {$ $} (66f)
(44f) edge node [right] {$ $} (55f)
(44f) edge node [right] {$ $} (66f)
(55f) edge node [right] {$ $} (66f)
(1f) edge node [right] {$ $} (11f)
(2f) edge node [right] {$ $} (22f)
(3f) edge node [right] {$ $} (33f)
(1f) edge node [right] {$ $} (44f)
(2f) edge node [right] {$ $} (55f)
(3f) edge node [right] {$ $} (66f);
\node (1) at (0,-1) {$c$};
\node (2) at (.5,-.5) {$a_3$};
\node (3) at (.5,-1.5) {$a_4$};
\node (4) at (1,-.75) {$a_1$};
\node (5) at (1,-1.25) {$a_2$};
\node (6) at (2.25,-.5) {$b_3$};
\node (7) at (2.25,-1.5) {$b_4$};
\node (8) at (2.75,-.75) {$b_1$};
\node (9) at (2.75,-1.25) {$b_2$};
\node (10) at (1.75,-1) {$b_5$};
\path[font=\small,>=angle 90]
(1) edge node [right] {$ $} (2)
(1) edge node [right] {$ $} (3)
(1) edge node [right] {$ $} (4)
(1) edge node [right] {$ $} (5)
(2) edge node [right] {$ $} (3)
(2) edge node [right] {$ $} (4)
(2) edge node [right] {$ $} (5)
(3) edge node [right] {$ $} (4)
(3) edge node [right] {$ $} (5)
(4) edge node [right] {$ $} (5)
(6) edge node [right] {$ $} (7)
(6) edge node [right] {$ $} (8)
(6) edge node [right] {$ $} (9)
(6) edge node [right] {$ $} (10)
(7) edge node [right] {$ $} (8)
(7) edge node [right] {$ $} (9)
(7) edge node [right] {$ $} (10)
(8) edge node [right] {$ $} (9)
(8) edge node [right] {$ $} (10)
(9) edge node [right] {$ $} (10)
(2) edge node [right] {$ $} (6)
(3) edge node [right] {$ $} (7)
(4) edge node [right] {$ $} (8)
(5) edge node [right] {$ $} (9)
(4) edge node [right] {$ $} (10);
\node (1d) at (4,-.5) {$a_1$};
\node (2d) at (4,-1.5) {$a_2$};
\node (3d) at (4.5,-.75) {$a_3$};
\node (4d) at (4.5,-1.25) {$a_4$};
\node (11d) at (5.75,-.5) {$b_1$};
\node (22d) at (5.75,-1.5) {$b_2$};
\node (33d) at (5.25,-.75) {$b_3$};
\node (44d) at (5.25,-1.25) {$b_4$};
\node (55d) at (6.25,-.75) {$b_5$};
\node (66d) at (6.25,-1.25) {$b_6$};
\node (77d) at (3.5,-1) {$c$};
\path[font=\small,>=angle 90]
(1d) edge node [right] {$ $} (2d)
(1d) edge node [right] {$ $} (3d)
(1d) edge node [right] {$ $} (4d)
(2d) edge node [right] {$ $} (3d)
(2d) edge node [right] {$ $} (4d)
(3d) edge node [right] {$ $} (4d)
(77d) edge node [right] {$ $} (1d)
(77d) edge node [right] {$ $} (2d)
(77d) edge node [right] {$ $} (3d)
(77d) edge node [right] {$ $} (4d)
(11d) edge node [right] {$ $} (22d)
(11d) edge node [right] {$ $} (33d)
(11d) edge node [right] {$ $} (44d)
(11d) edge node [right] {$ $} (55d)
(11d) edge node [right] {$ $} (66d)
(22d) edge node [right] {$ $} (33d)
(22d) edge node [right] {$ $} (33d)
(22d) edge node [right] {$ $} (44d)
(22d) edge node [right] {$ $} (55d)
(22d) edge node [right] {$ $} (66d)
(33d) edge node [right] {$ $} (44d)
(33d) edge node [right] {$ $} (55d)
(33d) edge node [right] {$ $} (66d)
(44d) edge node [right] {$ $} (55d)
(44d) edge node [right] {$ $} (66d)
(55d) edge node [right] {$ $} (66d)
(1d) edge node [right] {$ $} (11d)
(2d) edge node [right] {$ $} (22d)
(3d) edge node [right] {$ $} (33d)
(4d) edge node [right] {$ $} (44d)
(1d) edge node [right] {$ $} (55d)
(2d) edge node [right] {$ $} (66d);
\end{tikzpicture}
$
    \caption{Examples of graphs in the family $\{\SL{k}{n}\}$: $\SL{3}{2}$, $\SL{3}{3}$, $\SL{4}{1}$, and $\SL{4}{2}$}
    \label{figL}
\end{figure}
\begin{figure}[htb]
    \centering
$
\begin{tikzpicture}[scale=2]
\node (1a) at (.5,.5) {$a_1$};
\node (2a) at (0,1) {$a_2$};
\node (3a) at (0,0) {$a_3$};
\node (11a) at (1.25,.75) {$b_1$};
\node (22a) at (1.75,1) {$b_2$};
\node (33a) at (1.75,0) {$b_3$};
\node (44a) at (1.25,.25) {$b_4$};
\node (55a) at (2.25,.5) {$c$};
\path[font=\small,>=angle 90]
(1a) edge node [right] {$ $} (2a)
(1a) edge node [right] {$ $} (3a)
(2a) edge node [right] {$ $} (3a)
(11a) edge node [right] {$ $} (22a)
(11a) edge node [right] {$ $} (33a)
(11a) edge node [right] {$ $} (44a)
(11a) edge node [right] {$ $} (55a)
(22a) edge node [right] {$ $} (33a)
(22a) edge node [right] {$ $} (44a)
(22a) edge node [right] {$ $} (55a)
(33a) edge node [right] {$ $} (44a)
(33a) edge node [right] {$ $} (55a)
(44a) edge node [right] {$ $} (55a)
(1a) edge node [right] {$ $} (11a)
(2a) edge node [right] {$ $} (22a)
(3a) edge node [right] {$ $} (33a)
(1a) edge node [right] {$ $} (44a);
\node (1b) at (3,1) {$a_1$};
\node (2b) at (3,0) {$a_2$};
\node (3b) at (3.5,.5) {$a_3$};
\node (11b) at (4.75,1) {$b_1$};
\node (22b) at (4.75,0) {$b_2$};
\node (33b) at (4.25,.5) {$b_3$};
\node (44b) at (5.25,.85) {$b_4$};
\node (55b) at (5.25,.15) {$b_5$};
\node (66b) at (5.75,.5) {$c$};
\path[font=\small,>=angle 90]
(1b) edge node [right] {$ $} (2b)
(1b) edge node [right] {$ $} (3b)
(2b) edge node [right] {$ $} (3b)
(11b) edge node [right] {$ $} (22b)
(11b) edge node [right] {$ $} (33b)
(11b) edge node [right] {$ $} (44b)
(11b) edge node [right] {$ $} (55b)
(11b) edge node [right] {$ $} (66b)
(22b) edge node [right] {$ $} (33b)
(22b) edge node [right] {$ $} (33b)
(22b) edge node [right] {$ $} (44b)
(22b) edge node [right] {$ $} (55b)
(22b) edge node [right] {$ $} (66b)
(33b) edge node [right] {$ $} (44b)
(33b) edge node [right] {$ $} (55b)
(33b) edge node [right] {$ $} (66b)
(44b) edge node [right] {$ $} (55b)
(44b) edge node [right] {$ $} (66b)
(55b) edge node [right] {$ $} (66b)
(1b) edge node [right] {$ $} (11b)
(2b) edge node [right] {$ $} (22b)
(3b) edge node [right] {$ $} (33b)
(1b) edge node [right] {$ $} (44b)
(2b) edge node [right] {$ $} (55b);
\node (1c) at (0,-.5) {$a_1$};
\node (2c) at (0,-1.5) {$a_2$};
\node (3c) at (.5,-.75) {$a_3$};
\node (4c) at (.5,-1.25) {$a_4$};
\node (11c) at (1.75,-.5) {$b_1$};
\node (22c) at (1.75,-1.5) {$b_2$};
\node (33c) at (1.25,-.75) {$b_3$};
\node (44c) at (1.25,-1.25) {$b_4$};
\node (55c) at (2.25,-.75) {$b_5$};
\node (66c) at (2.25,-1.25) {$c$};
\path[font=\small,>=angle 90]
(1c) edge node [right] {$ $} (2c)
(1c) edge node [right] {$ $} (3c)
(1c) edge node [right] {$ $} (4c)
(2c) edge node [right] {$ $} (3c)
(2c) edge node [right] {$ $} (4c)
(3c) edge node [right] {$ $} (4c)
(11c) edge node [right] {$ $} (22c)
(11c) edge node [right] {$ $} (33c)
(11c) edge node [right] {$ $} (44c)
(11c) edge node [right] {$ $} (55c)
(11c) edge node [right] {$ $} (66c)
(22c) edge node [right] {$ $} (33c)
(22c) edge node [right] {$ $} (33c)
(22c) edge node [right] {$ $} (44c)
(22c) edge node [right] {$ $} (55c)
(22c) edge node [right] {$ $} (66c)
(33c) edge node [right] {$ $} (44c)
(33c) edge node [right] {$ $} (55c)
(33c) edge node [right] {$ $} (66c)
(44c) edge node [right] {$ $} (55c)
(44c) edge node [right] {$ $} (66c)
(55c) edge node [right] {$ $} (66c)
(1c) edge node [right] {$ $} (11c)
(2c) edge node [right] {$ $} (22c)
(3c) edge node [right] {$ $} (33c)
(4c) edge node [right] {$ $} (44c)
(1c) edge node [right] {$ $} (55c);
\node (1d) at (3,-.5) {$a_1$};
\node (2d) at (3,-1.5) {$a_2$};
\node (3d) at (3.5,-.75) {$a_3$};
\node (4d) at (3.5,-1.25) {$a_4$};
\node (11d) at (4.75,-.5) {$b_1$};
\node (22d) at (4.75,-1.5) {$b_2$};
\node (33d) at (4.25,-.75) {$b_3$};
\node (44d) at (4.25,-1.25) {$b_4$};
\node (55d) at (5.25,-.65) {$b_5$};
\node (66d) at (5.25,-1.35) {$b_6$};
\node (77d) at (5.75,-1) {$c$};
\path[font=\small,>=angle 90]
(1d) edge node [right] {$ $} (2d)
(1d) edge node [right] {$ $} (3d)
(1d) edge node [right] {$ $} (4d)
(2d) edge node [right] {$ $} (3d)
(2d) edge node [right] {$ $} (4d)
(3d) edge node [right] {$ $} (4d)
(11d) edge node [right] {$ $} (22d)
(11d) edge node [right] {$ $} (33d)
(11d) edge node [right] {$ $} (44d)
(11d) edge node [right] {$ $} (55d)
(11d) edge node [right] {$ $} (66d)
(11d) edge node [right] {$ $} (77d)
(22d) edge node [right] {$ $} (33d)
(22d) edge node [right] {$ $} (33d)
(22d) edge node [right] {$ $} (44d)
(22d) edge node [right] {$ $} (55d)
(22d) edge node [right] {$ $} (66d)
(22d) edge node [right] {$ $} (77d)
(33d) edge node [right] {$ $} (44d)
(33d) edge node [right] {$ $} (55d)
(33d) edge node [right] {$ $} (66d)
(33d) edge node [right] {$ $} (77d)
(44d) edge node [right] {$ $} (55d)
(44d) edge node [right] {$ $} (66d)
(44d) edge node [right] {$ $} (77d)
(55d) edge node [right] {$ $} (66d)
(55d) edge node [right] {$ $} (77d)
(66d) edge node [right] {$ $} (77d)
(1d) edge node [right] {$ $} (11d)
(2d) edge node [right] {$ $} (22d)
(3d) edge node [right] {$ $} (33d)
(4d) edge node [right] {$ $} (44d)
(1d) edge node [right] {$ $} (55d)
(2d) edge node [right] {$ $} (66d);
\end{tikzpicture}
$
    \caption{Examples of graphs in the family $\{\SR{k}{n}\}$: $\SR{3}{1}$, $\SR{3}{2}$, $\SR{4}{1}$, and $\SR{4}{2}$}
    \label{figR}
\end{figure}
The goal of this paper will be to fully investigate the former family. Using the power of Lemma \ref{lemma3}, we will illustrate that precisely one graph of the form $\SL{k}{n}$ occurs as the prime character degree graph of a solvable group. In particular, we shall prove our main theorem, stated again below:

\begin{thmmain}
The graph $\SL{k}{n}$ occurs as the prime character degree graph of a solvable group only when $(k,n)=(1,1)$ (see Figure \ref{fig11L}). 
\end{thmmain}

\begin{figure}[htb]
    \centering
$
\begin{tikzpicture}[scale=2]
\node (0a) at (0,.5) {$c$};
\node (1a) at (.5,.5) {$a_1$};
\node (11a) at (1.25,1) {$b_1$};
\node (22a) at (1.25,0) {$b_2$};
\path[font=\small,>=angle 90]
(0a) edge node [right] {$ $} (1a)
(11a) edge node [right] {$ $} (22a)
(1a) edge node [right] {$ $} (11a)
(1a) edge node [right] {$ $} (22a);
\end{tikzpicture}
$
    \caption{The graph $\SL{1}{1}$}
    \label{fig11L}
\end{figure}
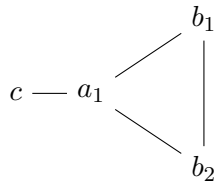

\section{The First Family of Graphs}\label{first}

The goal of this section is to prove that $\SL{k}{n}$ occurs as the prime character degree graph of a solvable group if and only if $(k,n) = (1,1)$. First, we note that it is easy to illustrate $\SL{1}{1}$ is the prime character degree graph of a solvable group; it can be realized as $\Gamma_{3,1}$ from \cite{BissLaub}. We shall show that if $(k,n) \neq (1,1)$, then $\SL{k}{n}$ is not the prime character degree graph of any solvable group. We will proceed by induction on $n$. It is worthwhile to note here that induction on $k$ is not necessary: $\SL{t}{n}$ does not depend on $\SL{k}{n}$ for integers satisfying $1 \leq n \leq t \leq k$. However, in inducting on $n$ and not $k$, we shall operate under the assumption that $k$ is large enough for the graph $\SL{k}{n}$ to be defineable, or simply that $k$, wherever denoted, is an integer satisfying $k \geq \max{\{2, n\}}$.

The bulk of this section's work consists of the proof of the following proposition, which serves as the base case for our induction:

\begin{proposition}\label{bc1}
Let $k\geq2$. The graph $\SL{k}{1}$ is not the prime character degree graph of any solvable group.
\end{proposition}

We will be proving Proposition \ref{bc1} by contradiction: that is, we will assume $\SL{k}{1} = \Delta (G)$ for some solvable $G$ with $|G|$ minimal. The basic idea of the proof is to illustrate that the Fitting subgroup of $G$ is necessarily minimal normal and then use Lemma \ref{lemma3} to arrive at a contradiction. However, to be able to apply Lemma \ref{lemma3}, we need to first establish a series of technical lemmas. One can refer to Figure \ref{figLk1} for examples of graphs in the base case.

\begin{figure}[htb]
    \centering
$
\begin{tikzpicture}[scale=2]
\node (0a) at (0,.5) {$c$};
\node (1a) at (.5,.5) {$a_1$};
\node (11a) at (1.25,1) {$b_1$};
\node (22a) at (1.25,0) {$b_2$};
\path[font=\small,>=angle 90]
(0a) edge node [right] {$ $} (1a)
(11a) edge node [right] {$ $} (22a)
(1a) edge node [right] {$ $} (11a)
(1a) edge node [right] {$ $} (22a);
\node (0b) at (2,.5) {$c$};
\node (1b) at (2.5,1) {$a_1$};
\node (2b) at (2.5,0) {$a_2$};
\node (33b) at (3.25,.5) {$b_3$};
\node (11b) at (3.75,1) {$b_1$};
\node (22b) at (3.75,0) {$b_2$};
\path[font=\small,>=angle 90]
(0b) edge node [right] {$ $} (1b)
(0b) edge node [above] {$ $} (2b)
(1b) edge node [above] {$ $} (2b)
(11b) edge node [above] {$ $} (22b)
(11b) edge node [above] {$ $} (33b)
(22b) edge node [above] {$ $} (33b)
(1b) edge node [above] {$ $} (11b)
(2b) edge node [above] {$ $} (22b)
(1b) edge node [above] {$ $} (33b);
\node (0c) at (4.5,.5) {$c$};
\node (2c) at (5,1) {$a_2$};
\node (3c) at (5,0) {$a_3$};
\node (1c) at (5.5,.5) {$a_1$};
\node (11c) at (6.25,.75) {$b_1$};
\node (44c) at (6.25,.25) {$b_4$};
\node (22c) at (6.75,1) {$b_2$};
\node (33c) at (6.75,0) {$b_3$};
\path[font=\small,>=angle 90]
(0c) edge node [right] {$ $} (1c)
(0c) edge node [above] {$ $} (2c)
(0c) edge node [above] {$ $} (3c)
(1c) edge node [above] {$ $} (2c)
(1c) edge node [above] {$ $} (3c)
(2c) edge node [right] {$ $} (3c)
(11c) edge node [above] {$ $} (22c)
(11c) edge node [right] {$ $} (33c)
(11c) edge node [above] {$ $} (44c)
(22c) edge node [above] {$ $} (33c)
(22c) edge node [above] {$ $} (44c)
(33c) edge node [above] {$ $} (44c)
(1c) edge node [above] {$ $} (11c)
(2c) edge node [above] {$ $} (22c)
(3c) edge node [above] {$ $} (33c)
(1c) edge node [above] {$ $} (44c);
\node (1d) at (0,-1) {$c$};
\node (2d) at (.5,-.5) {$a_3$};
\node (3d) at (.5,-1.5) {$a_4$};
\node (4d) at (1,-.75) {$a_1$};
\node (5d) at (1,-1.25) {$a_2$};
\node (6d) at (2.25,-.5) {$b_3$};
\node (7d) at (2.25,-1.5) {$b_4$};
\node (8d) at (2.75,-.75) {$b_1$};
\node (9d) at (2.75,-1.25) {$b_2$};
\node (10d) at (1.75,-1) {$b_5$};
\path[font=\small,>=angle 90]
(1d) edge node [right] {$ $} (2d)
(1d) edge node [right] {$ $} (3d)
(1d) edge node [right] {$ $} (4d)
(1d) edge node [right] {$ $} (5d)
(2d) edge node [right] {$ $} (3d)
(2d) edge node [right] {$ $} (4d)
(2d) edge node [right] {$ $} (5d)
(3d) edge node [right] {$ $} (4d)
(3d) edge node [right] {$ $} (5d)
(4d) edge node [right] {$ $} (5d)
(6d) edge node [right] {$ $} (7d)
(6d) edge node [right] {$ $} (8d)
(6d) edge node [right] {$ $} (9d)
(6d) edge node [right] {$ $} (10d)
(7d) edge node [right] {$ $} (8d)
(7d) edge node [right] {$ $} (9d)
(7d) edge node [right] {$ $} (10d)
(8d) edge node [right] {$ $} (9d)
(8d) edge node [right] {$ $} (10d)
(9d) edge node [right] {$ $} (10d)
(2d) edge node [right] {$ $} (6d)
(3d) edge node [right] {$ $} (7d)
(4d) edge node [right] {$ $} (8d)
(5d) edge node [right] {$ $} (9d)
(4d) edge node [right] {$ $} (10d);
\node (1) at (3.5,-1) {$c$};
\node (2) at (4,-.5) {$a_2$};
\node (3) at (4,-1.5) {$a_3$};
\node (4) at (4.5,-.6) {$a_4$};
\node (5) at (4.5,-1.4) {$a_5$};
\node (6) at (5,-1) {$a_1$};
\node (7) at (5.75,-.8) {$b_1$};
\node (8) at (5.75,-1.2) {$b_6$};
\node (9) at (6.25,-.5) {$b_2$};
\node (10) at (6.25,-1.5) {$b_3$};
\node (11) at (6.75,-.6) {$b_4$};
\node (12) at (6.75,-1.4) {$b_5$};
\path[font=\small,>=angle 90]
(1) edge node [right] {$ $} (2)
(1) edge node [right] {$ $} (3)
(1) edge node [right] {$ $} (4)
(1) edge node [right] {$ $} (5)
(1) edge node [right] {$ $} (6)
(2) edge node [right] {$ $} (3)
(2) edge node [right] {$ $} (4)
(2) edge node [right] {$ $} (5)
(2) edge node [right] {$ $} (6)
(3) edge node [right] {$ $} (4)
(3) edge node [right] {$ $} (5)
(3) edge node [right] {$ $} (6)
(4) edge node [right] {$ $} (5)
(4) edge node [right] {$ $} (6)
(5) edge node [right] {$ $} (6)
(7) edge node [right] {$ $} (8)
(7) edge node [right] {$ $} (9)
(7) edge node [right] {$ $} (10)
(7) edge node [right] {$ $} (11)
(7) edge node [right] {$ $} (12)
(8) edge node [right] {$ $} (9)
(8) edge node [right] {$ $} (10)
(8) edge node [right] {$ $} (11)
(8) edge node [right] {$ $} (12)
(9) edge node [right] {$ $} (10)
(9) edge node [right] {$ $} (11)
(9) edge node [right] {$ $} (12)
(10) edge node [right] {$ $} (11)
(10) edge node [right] {$ $} (12)
(11) edge node [right] {$ $} (12)
(2) edge node [right] {$ $} (9)
(3) edge node [right] {$ $} (10)
(4) edge node [right] {$ $} (11)
(5) edge node [right] {$ $} (12)
(6) edge node [right] {$ $} (7)
(6) edge node [right] {$ $} (8);
\end{tikzpicture}
$
    \caption{Examples of graphs in the family $\{\SL{k}{1}\}$: $1\leq k\leq5$}
    \label{figLk1}
\end{figure}
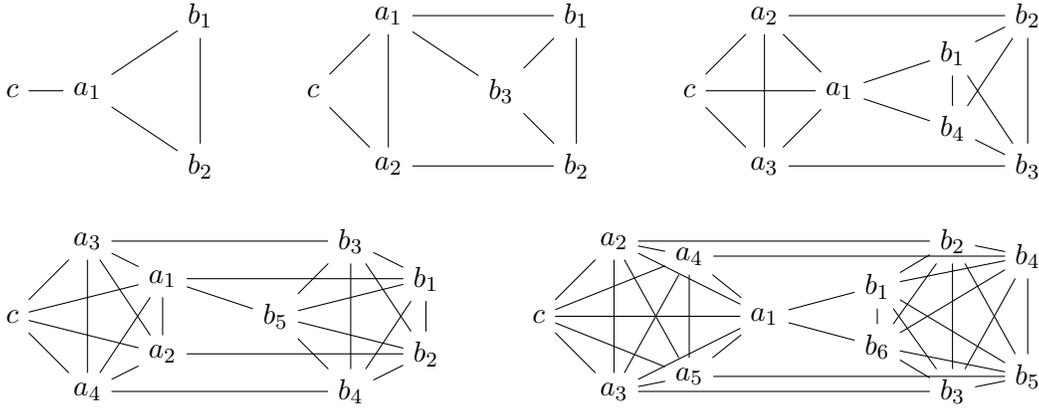

\subsection{Proof of Proposition \ref{bc1}} \label{lbaseproof}

Under the assumption that $\SL{k}{1} = \Delta(G)$ for some solvable $G$ with $|G|$ minimal, we will first show that there are no normal nonabelian Sylow subgroups associated to any of the vertices in $\Delta(G)$. Lemmas \ref{lemma<1}, \ref{lemma<2}, and \ref{lemma<3} handle the vertex sets $\{a_1, a_2, \ldots, a_k, b_1, b_{k+1}\}$, $\{b_2, \ldots, b_k\}$ and $\{c\}$, respectively. Finally, the Fitting subgroup of $G$ is addressed in Lemma \ref{minnorm}. 

\begin{lemma}\label{lemma<1}
Let $k\geq2$ and assume $\SL{k}{1}=\Delta(G)$ for some finite solvable group $G$, where $|G|$ is minimal. Then $G$ does not have a normal nonabelian Sylow $p$-subgroup for any $p\in\{a_1\}\cup\{a_2,\ldots,a_k\}\cup\{b_1,b_{k+1}\}$. In particular, $p$ is a strongly admissible vertex.
\end{lemma}
\begin{proof}
First we consider the vertex $a_1$. If we remove $a_1$ and all incident edges, we are left with a graph that has diameter three, violating Proposition \ref{sassresult}\eqref{sass1} in the case that $k=2$ or \eqref{sass3} if $k \geq 3$. Next, removing the edge between $a_1$ and $a_i$ (with $2\leq i \leq k$) or the edge between $a_1$ and $c$ results in an odd cycle in the complement graph, which cannot happen due to Theorem \ref{et}. Finally, losing one or both of the edges between $a_1$ and $b_1$ or $a_1$ and $b_{k+1}$ again yields a graph with diameter three. In particular, $|\rho_1\cup\rho_2|=k+1$ and $|\rho_3\cup\rho_4|=k+1$, which will violate Proposition \ref{sassresult}\eqref{sass3}. Thus, $a_1$ is admissible.

To see that $a_1$ is in fact strongly admissible, we need to consider the graph obtained by removing $a_1$, all incident edges, and possibly the edges between $a_i$ and $a_j$, the edges between $a_i$ and $c$, and the edge between $b_1$ and $b_{k+1}$. Regardless, the resulting graph continues to be a connected graph with diameter three violating Proposition \ref{sassresult}. Hence $a_1$ is strongly admissible.

Next we consider the vertex $a_2$. If we remove $a_2$ and all incident edges, we are left with a graph which has diameter three. Notice that $\left|\rho_1\cup \rho_2\right|=k$ and $\left|\rho_3\cup\rho_4 \right|=k+1$. Thus this graph does not occur because it contradicts Proposition \ref{sassresult}\eqref{sass3}. Next, removing the edge between $a_2$ and $a_i$ (with $1\leq i \leq k$ and $i\neq2$) or the edge between $a_2$ and $c$ violates Theorem \ref{et}. Finally, losing the edge between $a_2$ and $b_2$ yields a graph with diameter three which violates Proposition \ref{sassresult}\eqref{sass3} since $\left|\rho_1\cup\rho_2\right|=k+1$ and $\left|\rho_3\cup\rho_4\right|=k+1$. Thus, $a_2$ is admissible.

To see that $a_2$ is strongly admissible, we need to consider the graph obtained by removing $a_2$, all incident edges, and possibly the edges between $a_i$ and $a_j$ and the edges between $a_i$ and $c$. Regardless, the resulting graph continues to be a connected graph with diameter three violating Proposition \ref{sassresult}. Hence $a_2$ is strongly admissible, and by symmetry so too are the vertices $a_3,\ldots,a_k$.

Finally we consider the vertex $b_1$. If we remove $b_1$ and all incident edges, we are left with the graph $\Gamma_{k+1,k}$. This is shown not to occur by Theorem \ref{KT}. Next, removing the edge between $b_1$ and $b_j$ (with $2 \leq j \leq k+1$) violates P\'alfy's condition with $b_1$, $b_j$, and $c$. Therefore, we cannot lose the edge between $b_1$ and $b_j$. Lastly, losing the edge between $b_1$ and $a_1$ was considered above. Thus, $b_1$ is admissible.

To see that $b_1$ is indeed strongly admissible, we need to consider the graph obtained by removing $b_1$, all incident edges, and possibly the edges between $b_i$ and $b_j$, and the edge between $b_{k+1}$ and $a_1$. Regardless, the resulting graph cannot occur due to either Theorem \ref{KT} or Corollary \ref{subKT}. Hence $b_1$ is strongly admissible, and by symmetry so is $b_{k+1}$.

Hence by Lemma \ref{strong}, we get that there is no normal nonabelian Sylow $p$-subgroup for any $p\in\{a_1\}\cup\{a_2,\ldots,a_k\}\cup\{b_1,b_{k+1}\}$.
\end{proof}

We next address the remaining vertices of the subgraph $B$ (as defined in the beginning of Section \ref{gfamilies}).

\begin{lemma}\label{lemma<2}
Let $k\geq2$ and assume $\SL{k}{1}=\Delta(G)$ for some finite solvable group $G$, where $|G|$ is minimal. Then $G$ does not have a normal nonabelian Sylow $q$-subgroup for any $q\in\{b_2,\ldots,b_k\}$.
\end{lemma}
\begin{proof}
We shall show this for $q=b_2$. Following the notation of Lemma \ref{pi}, we have $\pi=\{a_2,b_1,\ldots,b_{k+1}\}$ and $\rho=\{c,a_1,a_3,\ldots,a_k\}$. We further take $\pi_1=\{a_2\}$ and $\pi_2=\{b_1,b_3,\ldots,b_{k+1}\}$, as well as $v=b_1$, $s=a_1$, and $w=c$. Observe $a_1$ was shown to be admissible in Lemma \ref{lemma<1}. Thus, by Lemma \ref{pi}, $G$ has no normal nonabelian Sylow $b_2$-subgroup.

One can replicate this argument for all remaining $q\in \{b_3,\ldots,b_k\}$ by symmetry, whence the result.
\end{proof}

Finally we address the special vertex $c$, which requires employment of the technical hypothesis found in \cite{BLL}.

\begin{lemma}\label{lemma<3}
Let $k\geq2$ and assume $\SL{k}{1}=\Delta(G)$ for some finite solvable group $G$, where $|G|$ is minimal. Then $G$ does not have a normal nonabelian Sylow $c$-subgroup.
\end{lemma}
\begin{proof}
We will show that $\SL{k}{1}$ satisfies Hypothesis \ref{newhype} for all $k\geq2$ with $p=c$. Following the conventional notation, notice that $\pi=\{a_1,\ldots,a_k\}$ and $\rho=\{b_1,\ldots,b_{k+1}\}$.

For \eqref{11}, for each $b_i$ ($1\leq i\leq k$), notice that $a_j\in\pi$ ($1\leq j \leq k$ and $j\neq i$) is nonadjacent to $b_i$. For $b_{k+1}$, observe that $a_2$ is nonadjacent. Thus, \eqref{11} is satisfied.

For \eqref{22}, for each $a_i$ ($1\leq i\leq k$), notice that $b_j\in\rho$ ($1\leq j\leq k$ and $j\neq i$) is nonadjacent to $a_i$. Hence, \eqref{22} is satisfied.

For \eqref{33} we must start by showing that all vertices in $\pi$ are admissible. This was done in Lemma \ref{lemma<1}.
Next we must show that no proper connected graph with vertex $\{p\}\cup\pi^*\cup\rho$ occurs as the prime character degree graph of any solvable group. Again observe by Lemma \ref{lemma<1} that the vertices $a_2,\ldots,a_k$ are symmetric. Without loss of generality, set $\pi_i^*:=\{a_2,\ldots,a_i\}$ for $2\leq i\leq k$. Therefore, possible subsets of $\pi^*$ are the following: (a) $\{a_1\}$, (b) $\pi_i^*$, (c) $\{a_1\}\cup\pi_i^*$, and (d) $\pi$.

For (a), this leaves us with a graph of diameter three such that $|\rho_3|\leq2$. This contradicts Proposition \ref{sassresult}\eqref{sass1}. 

For (b), we are left with a graph with diameter three which will violate Proposition \ref{sassresult}\eqref{sass3}.

For (c), we need only consider $\{a_1\}\cup\pi_i^*$ for $2\leq i\leq k-1$ since the case for $i=k$ is considered in (d) below. When $\pi^*=\{a_1\}\cup\pi_i^*$, we have a connected graph of diameter three. Denote this graph as $\Delta(H)$ where we suppose $H$ is a subgroup of $G$. Observe that Proposition \ref{sassresult}\eqref{sass2} forces $\rho_3=\{b_1,\ldots,b_i,b_{k+1}\}$, and then by Proposition \ref{sassresult}\eqref{sass4} there exists a normal Sylow $p_1$-subgroup for exactly one prime $p_1\in\rho_3$. Set this normal Sylow $p_1$-subgroup as $P_1$. The resulting graph $\Delta(H/P_1')$, as per Lemma \ref{lewisgraph}, has vertex set $\rho(H/P_1')=\rho(H)\setminus\{p_1\}$ and is obtained from the graph of $\Delta(H)$ by removing the vertex $p_1$, all incident edges, and possibly edges between vertices adjacent to $p_1$. Regardless of the $p_1$, the graph $\Delta(H/P_1')$ is connected and has diameter three. Now we note that for $\Delta(H/P_1')$, Proposition \ref{sassresult}\eqref{sass2} gives $\rho_3=\{b_1,\ldots,b_i,b_{k+1}\}\setminus\{p_1\}$, and by Proposition \ref{sassresult}\eqref{sass4} we know that $H/P_1'$ has a normal Sylow $p_2$-subgroup for exactly one prime $p_2\in\rho_3$. Call this normal Sylow subgroup $P_2$. Notice that one can show that $P_1P_2=P_1\times P_2$ is normal in $H$, which also shows that $P_2$ is a normal Sylow $p_2$-subgroup of $H$, where $p_2\in\rho_3$. This contradicts Proposition \ref{sassresult}\eqref{sass4} since we now have two such subgroups: $P_1$ and $P_2$.

For (d), removing an edge between two vertices in $\rho$ or two vertices in $\pi$ would violate Theorem \ref{et}. Therefore these edges cannot be lost. If we remove the edge between $a_1$ and $b_1$, or $a_1$ and $b_{k+1}$, or any combination of edges which span from $a_i$ to $b_i$ ($2\leq i\leq k$), this leaves us with a diameter three graph which fails to occur as it contradicts Proposition \ref{sassresult}\eqref{sass3} since $\left|\rho_1\cup\rho_2\right|=k+1$ and $\left|\rho_3\cup\rho_4\right|=k+1$.

Thus, no proper connected subgraph with vertex set $\{p\}\cup\pi^*\cup\rho$ occurs as the prime character degree graph of any solvable group, and \eqref{33} is satisfied.

For \eqref{44}, notice the only cases where $\tau\neq\varnothing$ is when $\pi^*\cup\rho^*$ is one of the following: (a) $\{a_1\}\cup\{b_1\}$, or (b) $\{a_1\}\cup\{b_{k+1}\}$.

For (a), we would have that $\tau = \{b_{k+1}\}$. Thus, we must show that $b_{k+1}$ is admissible, which was shown above in Lemma \ref{lemma<1}. We shall then show that no proper connected subgraph with the vertex set $\rho(G)\setminus\tau^* = \rho(G)\setminus\{b_{k+1}\}$ occurs as $\Delta(G)$. If we remove $b_{k+1}$ and all incident edges, we are left with a (proper) connected subgraph of $\Gamma_{k+1,k}$ with the same vertex set, which is shown not to occur by either Theorem \ref{KT} or Corollary \ref{subKT}.

For (b), one can employ a symmetric argument as above, and therefore \eqref{44} is satisfied.

For \eqref{55}, a disconnected subgraph with vertex set $\rho(G)$ must have one component with vertex set $\{p\}\cup\pi$ and the other $\rho$. Notice that each component has order $k+1$. Observe $k+1<2^{k+1}-1$ for all $k\geq2$. Thus, this violates P\'alfy's inequality from \cite{P2} and \eqref{55} is satisfied.

Hence, the graph $\SL{k}{1}$ satisfies Hypothesis \ref{newhype}, and therefore $G$ has no normal nonabelian Sylow $c$-subgroup by way of Theorem \ref{newmain}.
\end{proof}

Notice that we have now shown that $\SL{k}{1}$ has no normal nonabelian Sylow $p$-subgroups for all $k\geq2$. Letting $F$ be the Fitting subgroup of $G$, we note that $\rho(G)=\pi(|G:F|)$. Therefore, $\rho(G)=\rho(G/\Phi(G))$, where $\Phi(G)$ is the Frattini subgroup of $G$. However, Lemma \ref{lemma<3} verifies Hypothesis \ref{newhype}\eqref{33} that there is no proper connected subgraph of $\Delta(G)$ that occurs, and also verifies Hypothesis \ref{newhype}\eqref{55} that the disconnected subgraph does not occur. Finally, since $|G|$ is minimal, we must have that $\Phi(G)=1$. We then apply Lemma III 4.4 of \cite{H} to see that there exists a subgroup $H$ of $G$ such that $G=HF$ and $H\cap F=1$.

\begin{lemma}\label{minnorm}
Let $k\geq2$ and assume $\SL{k}{1}=\Delta(G)$ for some finite solvable group $G$, where $|G|$ is minimal. Then the Fitting subgroup of $G$, denoted $F$, is minimal normal in $G$.
\end{lemma}
\begin{proof}
Following the notation of Lemma \ref{lemma<3}, we set $p=c$, $\pi=\{a_1,\ldots,a_k\}$, and $\rho=\{b_1,\ldots,b_{k+1}\}$. Let $F$ be the Fitting subgroup of $G$ and let $E$ be the Fitting subgroup of $H$, where $H$ is as above.

We proceed by contradiction; suppose there exists a normal subgroup $N$ of $G$ such that $1<N<F$. By Theorem III 4.5 of \cite{H}, there exists a normal subgroup $M$ of $G$ such that $F=N\times M$. Furthermore, because both $N$ and $M$ are nontrivial, we must have $\rho(G/N)\subset\rho(G)$ and $\rho(G/M)\subset\rho(G)$. For any vertex $q\in\rho(G)\setminus\rho(G/N)$, we know that $G/N$ must have a normal abelian Sylow $q$-subgroup. The class of finite groups with a normal abelian Sylow $q$-subgroup is a formation, and so $q$ must lie in $\rho(G/M)$. Hence, $\rho(G)=\rho(G/N)\cup\rho(G/M)$.

Following the arguments in \cite{BL} or \cite{L}, one considers $q\in\rho(G)\setminus\rho(G/N)$. Then $q\notin\rho(G/F)=\rho(H)$, and so $E$ must contain the Sylow $q$-subgroup of $H$. Since $q\in\rho(G)$, it then follows that $q$ divides $|H|$, and therefore $q$ divides $|E|$ as well. Since $\cd(G)$ contains a degree divisible by all the prime divisors of $|EF:F|=|E|$, one gleans that $\rho(G)\setminus(\rho(G/N)\cap\rho(G/M))$ lies in a complete subgraph of $\Delta(G)$. Thus $\rho(G)\setminus(\rho(G/N)\cap\rho(G/M))$ must lie in one of the following subsets: (a) $\{p\}\cup\pi$, (b) $\{a_i,b_i\}$ for $2\leq i\leq k$, (c) $\{a_1,b_1,b_{k+1}\}$, or (d) $\rho$.

We start by supposing that (a) occurs. This implies that $\rho\subseteq\rho(G/N)\cap\rho(G/M)$. Since $\rho(G/N)\cup\rho(G/M)=\rho(G)$, we have that $p$ must belong to at least one of $\rho(G/N)$ or $\rho(G/M)$. Without loss of generality, say $p\in\rho(G/N)$ and therefore $\{p\}\cup\rho\subseteq\rho(G/N)$. There are then two cases to consider for $\rho(G/N)$: (i) $\{p\}\cup\rho$ or (ii) $\{p\}\cup\pi^*\cup\rho$, where $\pi^*$ is a nonempty proper subset of $\pi$. For (i), if $\{p\}\cup\rho=\rho(G/N)$, then the only possible graph arising is the disconnected graph with components $\{p\}$ and $\rho$ as the vertex sets. By Theorem 5.5 from \cite{L2}, we have that $G/N$ has a central Sylow $a_i$-subgroup for some $1\leq i\leq k$. This implies that $O^{a_i}(G)<G$, which is a contradiction since $O^{a_i}(G)=G$ for all such $i$ since $a_i$ is admissible. This leaves us with (ii), and so we must have $\{p\}\cup\pi^*\cup\rho=\rho(G/N)$. By Lemma \ref{lemma<3}, we know that no connected graph with vertex set $\{p\}\cup\pi^*\cup\rho$ occurs (since Hypothesis \ref{newhype}\eqref{33} was satisfied). Therefore, the only option is for a disconnected graph to occur. Again by Theorem 5.5 from \cite{L2}, $G/N$ has a central Sylow $a_i$ subgroup for some $1\leq i\leq k$, where $a_i\in\pi\setminus\pi^*$. We then have $O^{a_i}(G)<G$ which is a contradiction as $a_i$ is admissible and $O^{a_i}(G)=G$.

Suppose (b) occurs. This implies that $\rho(G)\setminus\{a_i,b_i\}\subseteq\rho(G/N)\cap\rho(G/M)$. Similar to the above, we can let $b_i\in\rho(G/N)$ without loss of generality. Since $\rho(G/N)$ and $\rho(G/M)$ are both proper in $\rho(G)$, we know that $a_i\notin\rho(G/N)$. In particular, we have that $\rho(G/N)=\rho(G)\setminus\{a_i\}=\{p\}\cup\pi^*\cup\rho$ for the specific $\pi^*=\pi\setminus\{a_i\}$. As above, we know that no connected graph with this vertex set occurs, and therefore the only option is for a disconnected graph to occur. Again by Theorem 5.5 from \cite{L2}, we have that $G/N$ has a central Sylow $a_i$-subgroup and thus $O^{a_i}(G)<G$. This is a contradiction since $O^{a_i}(G)=G$ because $a_i$ is admissible.

Next we suppose that (c) occurs. This implies that $\rho(G)\setminus\{a_1,b_1,b_{k+1}\}\subseteq\rho(G/N)\cap\rho(G/M)$. Without loss of generality, we can say that $a_1\in\rho(G/N)$, and this then gives us three cases for $\rho(G/N)$, which again we know is proper in $\rho(G)$: (i) $\rho(G)\setminus\{b_1,b_{k+1}\}$, (ii) $\rho(G)\setminus\{b_1\}$, or (iii) $\rho(G)\setminus\{b_{k+1}\}$. For (i), we then know that $\rho(G/M)=\rho(G)\setminus\{a_1\}$ because $\rho(G)=\rho(G/N)\cup\rho(G/M)$. Indeed $\rho(G/M)=\{p\}\cup\pi^*\cup\rho$ for $\pi^*=\pi\setminus\{a_1\}$. Again by Lemma \ref{lemma<3} and Hypothesis \ref{newhype}\eqref{33}, we know that no connected graph with that vertex set can occur. Considering the disconnected graph with that vertex set, we can apply Theorem 5.5 from \cite{L2} to conclude that $G/M$ has a central Sylow $a_1$-subgroup. This implies $O^{a_1}(G)<G$ which is a contradiction since $a_1$ is admissible. For (ii), if $\rho(G/N)=\rho(G)\setminus\{b_1\}$, then in fact $\rho(G/N)=\rho(G)\setminus\tau^*$ for $\tau=\{b_1\}$. By Lemma \ref{lemma<3} and Hypothesis \ref{newhype}\eqref{44}, we know that no connected graph with this vertex set occurs as the prime character degree graph of any solvable group. For the disconnected graph, we again employ Theorem 5.5 from \cite{L2} to get that $G/N$ has a central Sylow $b_1$-subgroup. Hence $O^{b_1}(G)<G$. However $b_1$ was shown to be admissible in Lemma \ref{lemma<1} and therefore $O^{b_1}(G)=G$, a contradiction. The argument for (iii) is symmetric to that of (ii).

Finally suppose (d) occurs. We follow the argument in \cite{Lewis} in a generalized form. We get that $E$ has a Hall $\rho$-subgroup of $H$. Since it is known that $\cd(G)$ has a character degree divisible by all the primes dividing $|E|$, we get that $|E|$ is divisible by only those primes in $\rho$. Therefore, $E$ is the Hall $\rho$-subgroup of $H$. Next we find a character $\chi\in\Irr(G)$ such that all the primes in $\rho$ divide $\chi(1)$. Next, taking $\theta$ as an irreducible constituent of $\chi_{FE}$, we observe that $\chi(1)/\theta(1)$ divides $|G:FE|$ and $\chi(1)$ is relatively prime to $|G:FE|$. We conclude that $\chi_{FE}=\theta$. Next, since all the primes in $\rho$ divide $\theta(1)$, and the only possible prime divisors of $d\in\cd(G/FE)$ are those primes in $\{p\}\cup\pi$, we apply Gallagher's Theorem to conclude that $\cd(G/FE)=\{1\}$ and $G/FE$ is abelian. We now have that $O^{a_i}(G)<G$ for $1\leq i\leq k$, which is a contradiction since $a_i$ was shown to be admissible in Lemma \ref{lemma<1} and therefore $O^{a_i}(G)=G$.

Thus, no such $N$ can occur, and therefore we get that $F$ is minimal normal in $G$, which was what we wanted.
\end{proof}

We are finally able to finish the proof of Proposition \ref{bc1}:

\begin{proof}
For the sake of contradiction, suppose that $G$ is a counterexample with $|G|$ minimal such that $\Delta(G)=\SL{k}{1}$. Observe that the Fitting subgroup of $G$ is minimal normal in $G$ by Lemma \ref{minnorm}. Next, notice that the hypotheses from Lemma \ref{lemma3} are satisfied by taking $a=b_1$, $b=b_2$, $c=a_1$, and $d=a_2$, where we observe that $a_1$ and $a_2$ are admissible by Lemma \ref{lemma<1}. Applying Lemma \ref{lemma3} yields our contradiction, and therefore the graph $\SL{k}{1}$ is not the prime character degree graph of any solvable group. This completes the base case.
\end{proof}

\subsection{Proof of Theorem \ref{firstfam}}

Recall that $k\geq\max{\{2,m\}}$. Next assume our inductive hypothesis below:

\begin{hyp}\label{inhype}
Given any integer $m\geq1$, we assume that the graph $\SL{k}{m}$ does not occur as the prime character degree graph of any solvable group.
\end{hyp}

We then proceed to the inductive step:

\begin{proposition}\label{instep}
The graph $\SL{k}{m+1}$ does not occur as the prime character degree graph of any solvable group.
\end{proposition}
\begin{proof}
For the sake of contradiction, suppose $\SL{k}{m+1}=\Delta(G)$ for some finite solvable group $G$, where $|G|$ is minimal.

One can first replicate Lemma \ref{lemma<1} and show that $p$ is a strongly admissible vertex for all $p\in\{a_1\ldots,a_{m+1}\}\cup\{a_{m+2},\ldots,a_k\}\cup\{b_1,\ldots,b_{m+1},b_{k+1},\ldots,b_{k+m+1}\}$, which relies on Theorem \ref{et}, Proposition \ref{sassresult}, and Hypothesis \ref{inhype}. Thus, by Lemma \ref{strong}, $G$ does not have a normal nonabelian Sylow $p$-subgroup. Next, one can show that $G$ does not have a normal nonabelian Sylow $q$-subgroup for all $q\in\{b_{m+2},\dots,b_k\}$ by way of Lemma \ref{pi} (using an identical argument as that found in Lemma \ref{lemma<2}). Finally, showing that $\Delta(G)$ satisfies Hypothesis \ref{newhype}, similar to what was done in Lemma \ref{lemma<3}, guarantees that $G$ will have no normal nonabelian Sylow $c$-subgroup by employing Theorem \ref{newmain}.

Now that we know $G$ has no normal nonabelian Sylow $p$-subgroups, one can then follow a similar process as outlined in Lemma \ref{minnorm} and get that the Fitting subgroup $F$ is minimal normal in $G$. Therefore, taking $a=b_1$, $b=b_2$, $c=a_1$, and $d=a_2$ from Lemma \ref{lemma3} yield the contradiction.
\end{proof}

To conclude, we again state Theorem \ref{firstfam}:

\begin{thmmain}
The graph $\SL{k}{n}$ occurs as the prime character degree graph of a solvable group only when $(k,n)=(1,1)$ (see Figure \ref{fig11L}).
\end{thmmain}
\begin{proof}
Follows by the induction argument given through Proposition \ref{bc1}, Hypothesis \ref{inhype}, and Proposition \ref{instep}.
\end{proof}

\begin{corollary}
For all integers $1\leq n\leq k$, any proper connected subgraph of $\SL{k}{n}$ that has the same vertex set is not the prime character degree graph of any solvable group.
\end{corollary}
\begin{proof}
First observe that this holds for $\SL{1}{1}$. For $(k,n)\neq(1,1)$, notice that $\SL{k}{n}$ satisfies Hypothesis \ref{newhype}\eqref{33} for $\pi^*=\pi$.
\end{proof}

\section{Observations and Generalizations}\label{last}

Having classified the family $\{\SL{k}{n}\}$, there are natural extensions possibly worthy of investigation. For example, one could define a family of graphs using one-to-three edge mappings, as opposed to the focus on one-to-two edge mappings in $\{\SL{k}{n}\}$ (see Figure \ref{fig1-3}). Naturally, one can then extend to one-to-$m$ edge mappings for any positive integer $m$.
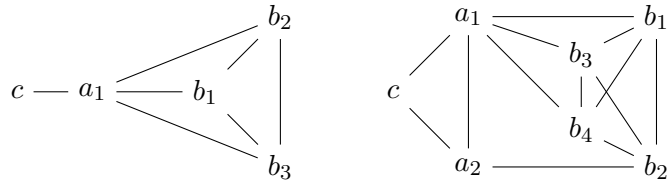
\begin{figure}[htb]
    \centering
$
\begin{tikzpicture}[scale=2]
\node (0a) at (0,.5) {$c$};
\node (1a) at (.5,.5) {$a_1$};
\node (11a) at (1.25,.5) {$b_1$};
\node (22a) at (1.75,1) {$b_2$};
\node (33a) at (1.75,0) {$b_3$};
\path[font=\small,>=angle 90]
(0a) edge node [right] {$ $} (1a)
(11a) edge node [right] {$ $} (22a)
(11a) edge node [right] {$ $} (33a)
(22a) edge node [right] {$ $} (33a)
(1a) edge node [right] {$ $} (11a)
(1a) edge node [right] {$ $} (22a)
(1a) edge node [right] {$ $} (33a);
\node (0b) at (2.5,.5) {$c$};
\node (1b) at (3,1) {$a_1$};
\node (2b) at (3,0) {$a_2$};
\node (11b) at (4.25,1) {$b_1$};
\node (22b) at (4.25,0) {$b_2$};
\node (33b) at (3.75,.75) {$b_3$};
\node (44b) at (3.75,.25) {$b_4$};
\path[font=\small,>=angle 90]
(0b) edge node [right] {$ $} (1b)
(0b) edge node [right] {$ $} (2b)
(1b) edge node [right] {$ $} (2b)
(11b) edge node [right] {$ $} (22b)
(11b) edge node [right] {$ $} (33b)
(11b) edge node [right] {$ $} (44b)
(22b) edge node [right] {$ $} (33b)
(22b) edge node [right] {$ $} (44b)
(33b) edge node [right] {$ $} (44b)
(1b) edge node [right] {$ $} (11b)
(2b) edge node [right] {$ $} (22b)
(1b) edge node [right] {$ $} (33b)
(1b) edge node [right] {$ $} (44b);
\end{tikzpicture}
$
    \caption{Examples of graphs with one-to-three edge mappings}
    \label{fig1-3}
\end{figure}
Even further, one may consider the possibility of having a mixture of one-to-one, one-to-two, and one-to-three edge mappings between the vertices of $A$ and the vertices of $B$. The next logical generalization would be to consider a graph with a mixture of one-to-one through one-to-four mappings, and so on. Presumably, one could use a starting point similar to the one used in this paper, where the case $n=1$ is handled and the remaining ones are addressed inductively. 

Of course, this generalization would complicate current notation as well as the pool of cases to consider. It is not altogether clear at the moment whether, after making the necessary adjustments in assumptions, the same approach used in this paper would be sufficient for handling this generalized version of $\{\SL{k}{n}\}$, but it would certainly rely on the family of graphs studied in Section \ref{first}.
Another direction to consider would be to introduce multiple ``special'' vertices that are adjacent only to vertices in either the left or right subgraph (see Figure \ref{fig2-3}).
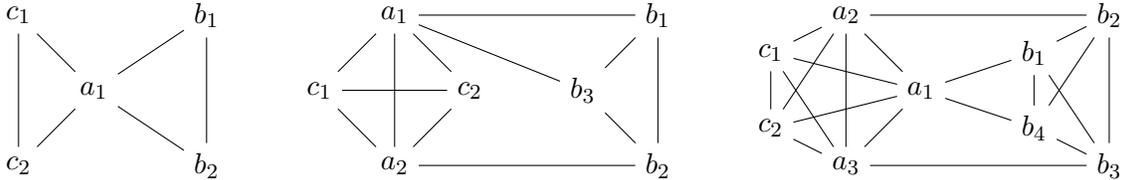
\begin{figure}[htb]
    \centering
$
\begin{tikzpicture}[scale=2]
\node (0a) at (0,1) {$c_1$};
\node (00a) at (0,0) {$c_2$};
\node (1a) at (.5,.5) {$a_1$};
\node (11a) at (1.25,1) {$b_1$};
\node (22a) at (1.25,0) {$b_2$};
\path[font=\small,>=angle 90]
(0a) edge node [right] {$ $} (00a)
(0a) edge node [right] {$ $} (1a)
(00a) edge node [right] {$ $} (1a)
(11a) edge node [right] {$ $} (22a)
(1a) edge node [right] {$ $} (11a)
(1a) edge node [right] {$ $} (22a);
\node (0b) at (2,.5) {$c_1$};
\node (00b) at (3,.5) {$c_2$};
\node (1b) at (2.5,1) {$a_1$};
\node (2b) at (2.5,0) {$a_2$};
\node (33b) at (3.75,.5) {$b_3$};
\node (11b) at (4.25,1) {$b_1$};
\node (22b) at (4.25,0) {$b_2$};
\path[font=\small,>=angle 90]
(0b) edge node [right] {$ $} (00b)
(0b) edge node [above] {$ $} (1b)
(0b) edge node [above] {$ $} (2b)
(00b) edge node [above] {$ $} (1b)
(00b) edge node [above] {$ $} (2b)
(1b) edge node [above] {$ $} (2b)
(11b) edge node [above] {$ $} (22b)
(11b) edge node [above] {$ $} (33b)
(22b) edge node [above] {$ $} (33b)
(1b) edge node [above] {$ $} (11b)
(2b) edge node [above] {$ $} (22b)
(1b) edge node [above] {$ $} (33b);
\node (0c) at (5,.75) {$c_1$};
\node (00c) at (5,.25) {$c_2$};
\node (2c) at (5.5,1) {$a_2$};
\node (3c) at (5.5,0) {$a_3$};
\node (1c) at (6,.5) {$a_1$};
\node (11c) at (6.75,.75) {$b_1$};
\node (44c) at (6.75,.25) {$b_4$};
\node (22c) at (7.25,1) {$b_2$};
\node (33c) at (7.25,0) {$b_3$};
\path[font=\small,>=angle 90]
(0c) edge node [right] {$ $} (00c)
(0c) edge node [above] {$ $} (1c)
(0c) edge node [above] {$ $} (2c)
(0c) edge node [above] {$ $} (3c)
(00c) edge node [above] {$ $} (1c)
(00c) edge node [above] {$ $} (2c)
(00c) edge node [above] {$ $} (3c)
(1c) edge node [above] {$ $} (2c)
(1c) edge node [above] {$ $} (3c)
(2c) edge node [right] {$ $} (3c)
(11c) edge node [above] {$ $} (22c)
(11c) edge node [right] {$ $} (33c)
(11c) edge node [above] {$ $} (44c)
(22c) edge node [above] {$ $} (33c)
(22c) edge node [above] {$ $} (44c)
(33c) edge node [above] {$ $} (44c)
(1c) edge node [above] {$ $} (11c)
(2c) edge node [above] {$ $} (22c)
(3c) edge node [above] {$ $} (33c)
(1c) edge node [above] {$ $} (44c);
\end{tikzpicture}
$
    \caption{Examples of graphs with two ``special" vertices}
    \label{fig2-3}
\end{figure}
It may also be worth considering whether all of these special vertices need to be associated to the same subgraph, or if we can allow a mixture in which some ``special'' vertices are in the left subgraph while others are in the right subgraph. Notably, as alluded to in Remark \ref{rmk}, this would immediately yield a graph with diameter three; one would need to alter the edge rules of the graphs. Moreover, this endeavor would likely rely on the classifications of both the families $\{\SL{k}{n}\}$ and $\{\SR{k}{n}\}$.

We intend to revisit the family of graphs $\{\SR{k}{n}\}$ in a follow-up paper, as the current technique of using subgraphs with diameter three does not sufficiently resolve the classification of the graphs $\{\SR{k}{n}\}$. It is not difficult to see that $\SR{1}{1}$ occurs as $\Delta(G)$ for some solvable group $G$ and $\SR{k}{1}$ does \emph{not} occur as $\Delta(G)$ for any solvable $G$ when $k \geq 3$. Observe, however, that the graph $\SR{2}{1}$ is one of the graphs in \cite{BL2} that is unclassified. In particular, it's currently unknown whether or not $\SR{2}{1}$ appears as $\Delta(G)$ for some solvable group $G$ despite being previously studied. 

For $n \geq 2$, the graph $\SR{k}{n}$ presents immediate obstacles to the current approach; a completely different technique is likely necessary.

\end{document}